\documentclass[a4paper,11pt]{article}
\title{On power values of pyramidal numbers, II }
\author{Andrej Dujella, K\'alm\'an Gy\H{o}ry, Philippe Michaud-Jacobs,  \\ and \'Akos Pint\'er}

\makeatletter
\newcommand\notsotiny{\@setfontsize\notsotiny\@vipt\@viipt}
\makeatother
\usepackage{amsmath} 
\usepackage{amssymb} 
\usepackage{mathtools}
\usepackage{framed}
\usepackage{pifont} 
\usepackage{enumerate} 
\usepackage{wasysym} 
\usepackage{commath}
\usepackage{graphicx}
\usepackage{framed}
\usepackage{amsmath}
\usepackage{amssymb}
\usepackage{amsfonts}
\usepackage{mathrsfs}
\usepackage{mathtools}
\usepackage{url}
\usepackage{array}
\usepackage{amsthm}
\usepackage[english]{babel}
\usepackage[utf8]{inputenc}
\newtheorem{theorem}{Theorem}[section]
\newtheorem*{theorem*}{Theorem}

\newtheorem {lemma}[theorem]{Lemma}
\newtheorem {proposition}[theorem]{Proposition}
\theoremstyle{definition}

\theoremstyle{remark}

\newtheorem*{remark*}{Remark}
\usepackage{yfonts}
\usepackage{amsmath,amssymb,tikz-cd}
\usepackage{pdflscape}

\usepackage{etoolbox}
\apptocmd{\sloppy}{\hbadness 10000\relax}{}{}

\newcommand{\ord}{\operatorname{ord}}

\usepackage[numbers]{natbib}
\usepackage{cite}
\usepackage{mathtools, nccmath}

\usepackage[section]{placeins}

\renewcommand{\arraystretch}{2.0}

\usepackage{longtable}
\usepackage{hyperref}

\usepackage{multirow}

\newcommand{\Addresses}{{
  \bigskip
  \footnotesize

 \textsc{Department of Mathematics, Faculty of Science, University of Zagreb, 
 Bijeni\v{c}ka cesta 30, 10000 Zagreb, Croatia }\par\nopagebreak
  \textit{E-mail address}: \texttt{duje@math.hr}
\bigskip

 \textsc{Institute of Mathematics, University of Debrecen, H-4002 Debrecen, P.O.B. 400, Hungary}\par\nopagebreak
  \textit{E-mail address}: \texttt{gyory@science.unideb.hu}
\bigskip

 \textsc{Mathematics Institute, University of Warwick, CV4 7AL, United Kingdom}\par\nopagebreak
  \textit{E-mail address}: \texttt{p.rodgers@warwick.ac.uk} \bigskip

\textsc{Institute of Mathematics, University of Debrecen, H-4002 Debrecen, P.O.B. 400 and MTA-DE Equations,
Functions and Curves Research Group, Eötvös Loránd Research Network (ELKH), Hungary}\par\nopagebreak
  \textit{E-mail address}: \texttt{apinter@science.unideb.hu}

}}

\let\svthefootnote\thefootnote
\newcommand\freefootnote[1]{%
  \let\thefootnote\relax%
  \footnotetext{#1}%
  \let\thefootnote\svthefootnote%
}

\date{\vspace{-3ex}}

\begin{document}

\maketitle

\begin{center}\emph{Dedicated to the memory of Professor Andrzej Schinzel} \end{center}
 
\begin{abstract}
For $m \geq 3$, we define the $m$th order pyramidal number by \[ \mathrm{Pyr}_m(x) = \frac{1}{6} x(x+1)((m-2)x+5-m). \] In a previous paper, written by the first-, second-, and fourth-named authors, all solutions to the equation $\mathrm{Pyr}_m(x) = y^2$ are found in positive integers $x$ and $y$, for $6 \leq m \leq 100$. In this paper, we consider the question of higher powers, and find all solutions to the equation $\mathrm{Pyr}_m(x) = y^n$ in positive integers $x$, $y$, and  $n$, with $n \geq 3$, and $5 \leq m \leq 50$. We reduce the problem to a study of systems of binomial Thue equations, and use a combination of local arguments, the modular method via Frey curves, and bounds arising from linear forms in logarithms to solve the problem.
\end{abstract}

\section{Introduction} 

For $m \geq 3$, we define the $m$th order pyramidal number by \[ \mathrm{Pyr}_m(x) = \frac{1}{6} x(x+1)((m-2)x+5-m). \] Pyramidal numbers are a special type of figurate number with many interesting properties and a rich history. The properties of figurate numbers, and in particular their relationship with perfect powers, have received much attention in the literature, see \citep{part1, poly, decagonal} for example, and the references therein. \freefootnote{\emph{Date}: \date{\today}.}
\freefootnote{\emph{Keywords}: Pyramidal, perfect power, binomial Thue equations, modular method, Frey curve, level-lowering.}
\freefootnote{\emph{MSC2010}: 11D41, 11D59, 11D61, 14G99.}
\freefootnote{A. D. is supported by the Croatian Science Foundation under the project no. IP-2018-011313
 and the QuantiXLie Center of Excellence, a project co-financed by the Croatian Government and 
 European Union through the European Regional Development Fund - the Competitiveness and Cohesion 
 Operational Programme (Grant KK.01.1.1.01.0004). }
\freefootnote{K. Gy. is supported in part by Grants K115479 and K128088 from 
 the Hungarian National Foundation for Scientific Research (OTKA) and from the Austrian-Hungarian 
joint project ANN 130909 (FWF-NKFIH).}
\freefootnote{P. M-J. is supported by an EPSRC studentship and has previously used the name Philippe Michaud-Rodgers.}
\freefootnote{\'A. P. is supported by Grants K115479 and K128088 from 
 the Hungarian National Foundation for Scientific Research (OTKA) and from the Austrian-Hungarian 
joint project ANN 130909 (FWF-NKFIH) and E\"otv\"os Lor\'and Research Network (ELKH).}

This paper is the second in a series of two papers. In the first paper \citep{part1}, written by the first-, second-, and fourth-named authors in 2012, the question of when $\mathrm{Pyr}_m(x)$ is a square is considered, and all solutions are found with $3 \leq m \leq 100$ and $m \ne 5$ (the cases $m = 3$ and $m=4$ are classical, and there are infinitely many solutions in the case $m=5$). In this paper, we consider the case of higher powers and obtain the following result.

\begingroup
\renewcommand\thetheorem{1}
\begin{theorem}\label{Mainthm} All the solutions of the equation
\begin{equation}\label{maineq} \mathrm{Pyr}_m(x)=\frac{1}{6}x(x+1)((m-2)x+5-m)=y^n \end{equation}
in positive integers $m,x,y,n$  with $y>1, \, 3\leq m\leq 50$ and $n \geq 3$ are 
\begin{align*} (m,x,y,n)= &  (5,57121,3107,4), \; (7,2,2,3), \;  (15,2,2,4), \; (17,8,6,4), \\ & (26,2,3,3), \; (31,2,2,5) \; \text{ and } \; (50,15,30,3). \end{align*}
\end{theorem}
\endgroup

When $m = 3$ or $m = 4$, these are known results (see \citep{m=3} and \citep{m=4}), and so we will consider the cases $5 \leq m \leq 50$. We note that there are no apparent obstructions to extending Theorem \ref{Mainthm} to larger values of $m$, say $m \leq 100$ for example, although we do not pursue this here as it would lead to needing to consider many further cases and would most likely not require new ideas.

The techniques we use to prove Theorem \ref{Mainthm} when $n $ is odd will be of a very different flavour to those used in \citep{part1}, where the main ideas are centred around the study of integral points on elliptic curves. Starting from equation (\ref{maineq}) we will form various systems of binomial Thue equations. We will then use a combination of local arguments  and the modular method via Frey curves, along with bounds arising from linear forms in logarithms to prove Theorem \ref{Mainthm}.

We now outline the rest of the paper. In Section 2, we use the results of \citep{part1} to find all solutions to equation (\ref{maineq}) in the case $n = 4$. We also treat the case $m = 5$. For the remainder of the paper, we will consider the case $n$ an odd prime and $ 6 \leq m \leq 50$. In Section 3, we reduce the problem to the study of finitely many systems of binomial Thue equations, and obtain a bound on $n$ using results on linear forms in logarithms. In Section 4, we solve the majority of these Thue equations using local arguments, and finally we use the modular method in various guises in Section 5 to deal with the remaining cases. 

\bigskip

The \texttt{Magma} \citep{magma} code used to support the computations in this paper can be found at: 

\vspace{3pt}

 \url{https://github.com/michaud-jacobs/pyramidal-2}

\bigskip

The third-named author would like to thank Samir Siksek and Damiano Testa for many useful discussions.

\section{The case \texorpdfstring{$n = 4$}{} and the case \texorpdfstring{$m = 5$}{}}

In this section we will treat certain cases that are not amenable to the methods of the later parts of the paper.

\begin{lemma}\label{n=4Lem} Let $(m,x,y)$ be a solution to equation (\ref{maineq}) with $n = 4$, $x > 0$, $y >1 $, and $6 \leq m \leq 50$. Then \[(m,x,y) = (15,2,2) \; \text{ or } \; (17,8,6). \]
\end{lemma}

\begin{proof} Let $(m,x,y)$ be such a solution to equation (\ref{maineq}). Then $(m,x,y^2)$ is a solution to equation (\ref{maineq}) with $n = 2$, and so we can apply the results of \citep[pp.~218--219]{part1} to obtain all solutions for $6 \leq m \leq 50$. The two solutions we find are those stated in the lemma.
\end{proof}
 
\begin{lemma}\label{m=5Lem} Let $m = 5$. Let $(x,y,n)$ be a solution to equation (\ref{maineq}) with $x > 0$, $y >1 $, and $n \geq 3$. Then \[(x,y,n) = (57121, 3107, 4). \]
\end{lemma} 
 
\begin{proof} Suppose $(x,y,n)$ is such a solution to equation (\ref{maineq}). We have \begin{equation}\label{m=5} x^2(x+1) = 2y^n. \end{equation} Suppose first that $n = 4$. Then $2\ord_2(x)+ \ord_2(x+1) = 1 + 4 \ord_2(y)$, so $x$ is odd and $x + 1$ is even. We have \[ x^2 \left (\frac{x+1}{2} \right) = y^4, \] and $\gcd(x^2, (x+1)/2) = 1$. It follows that $x^2 = y_1^4$ and $x+ 1 = 2y_2^4$ for some coprime integers $y_1$ and $y_2$ with $y = y_1y_2$. Then $x = y_1^2$, since $x$ is positive, so \begin{equation} y_1^2 + 1 = 2y_2^4. \end{equation} This equation is known as Ljunggren's equation, since Ljunggren proved that the only positive integer solutions to this equation are given by $(y_1,y_2) = (1,1)$ and $(y_1,y_2) = (239,13)$ (see \citep{ljunggren} for the original proof, which is somewhat involved, or \citep{simpler} for an example of a simpler proof). Since $y > 1$, we obtain the solution $(x,y) = (y_1^2, y_1y_2) = (57121,3107)$ to equation (\ref{maineq}).

Suppose instead that $n$ is odd. In this case $x$ may be odd or even. If $x$ is odd, then arguing similarly to above, there exist coprime integers $y_1$ and $y_2$ satisfying \begin{equation*} y_1^n + 1 = 2y^n. \end{equation*} This equation has no solutions with $y_1y_2 > 1$  and $n \geq 3$ by \citep[Main~Theorem]{ll2}. If $x$ is even, then we find that there exist coprime integers $y_1$ and $y_2$ satisfying \[ x^2 = 2y_1^n \; \text{ and } \; x+1 = y_2^n. \] From $x^2 = 2y_1^n$, we see that $\ord_2(y) = k \geq 1$, with $k$ odd, and we may write $y_1 = 2^k z_1^2$ for some positive integer $z_1$ coprime to $y_2$. Then  \[ x =  2^{\frac{kn+1}{2}} z_1^n, \] and so \[2^{\frac{kn+1}{2}} z_1^n - y_2^n = 1. \] By \citep[Theorem~1.2]{consecutive}, this equation has no solutions in positive coprime integers. This completes the case $m = 5$.
\end{proof}

\section{Systems of binomial Thue equations}

Thanks to Lemmas \ref{n=4Lem} and \ref{m=5Lem}, we may suppose that $n \geq 3$ is odd and that $6 \leq m \leq 50$. Moreover, we will assume that $n = p$ is an odd prime. We write equation (\ref{maineq}) as \begin{equation}\label{eq2} x(x+1)(a_mx-b_m) = c_m y^p, \end{equation} 
where \begin{align*} (a_m,b_m,c_m) = \begin{cases} \left(\frac{m-2}{3}, \frac{m-5}{3}, 2 \right) &  \text{ if } \; m \equiv 2 \pmod{3}, \\ (m-2, m-5, 6) & \text{ if } \;  m \not\equiv 2 \pmod{3}. \end{cases} \end{align*}
We introduce the notation \begin{equation*} d_1 = \gcd(x, a_mx-b_m) \quad \text{and} \quad d_2 = \gcd(x+1,a_mx-b_m). \end{equation*} By writing $a_mx-b_m = a_m(x+1) - (a_m+b_m)$, it is straightforward to see that \[ \gcd(x,x+1) = 1, \; \; d_1 \mid b_m, \; \; \text{ and} \; \; d_2 \mid a_m+b_m. \] Even though $d_1$ and  $d_2$ are unknown, we know a finite list of possibilities for each one; namely the divisors of $b_m$ and $a_m+b_m$ respectively. We also note that $\gcd(a_m,b_m) = 1$, since $m \ne 3$.

We may now divide both sides of equation (\ref{eq2}) by $c_m$ and the appropriate $p$th powers of $d_1$ and $d_2$ to obtain  \begin{equation} \label{denom} \left( \frac{x}{A} \right) \left( \frac{x+1}{B} \right) \left( \frac{a_mx-b_m}{C} \right) = Y^p.  \end{equation} Here, $Y \mid y$ is a positive integer,  $A, B,$ and $C$ are positive integers satisfying \[ A \mid c_m (b_m)^p, \quad B \mid c_m(a_m+b_m)^p, \quad C \mid c_m (b_m)^p (a_m+b_m)^p, \] and the three factors on the left-hand side of equation (\ref{denom}) are integral and pairwise coprime. Moreover, we may assume that $A$, $B$, and $C$ are $p$th power free. Later in this section we will provide, for each $m$, the precise list of possibilities for the triple $(A,B,C)$. We emphasise that the triple $(A,B,C)$ depends on $m$, and will usually also depend on $p$. It follows that there exist positive pairwise coprime integers $y_1, y_2,$ and $y_3$ satisfying
\begin{equation*} x = Ay_1^p, \qquad x+1 = By_2^p, \qquad a_mx-b_m = Cy_3^p. \end{equation*} This leads us to consider the following system of binomial Thue equations 
\begin{align} \label{system}
\begin{split}
B \, y_2^p - A \, y_1^p & = 1 \\
a_m A \, y_1^p - C \, y_3^p & = b_m. 
\end{split}
\end{align}
For each fixed value of $m$, we aim to solve this system of equations for each possible triple $(A,B,C)$. We start by providing a bound for $p$ using results on linear forms in logarithms.

\begin{proposition}\label{LogBound} Let $(x,y,m,p)$ be a solution to equation (\ref{eq2}) with $y > 1$. Then \[ p < 10676 \cdot \log \left( c_m^2 \cdot b_m \cdot (a_m+b_m) \right). \] 
\end{proposition}

\begin{proof} We aim to obtain a binomial Thue equation with coefficients independent of $p$. We divide both sides of equation (\ref{eq2}) by $c_m \cdot d_1^p \cdot d_2^p$. to obtain  \[ \left(\frac{x}{e_1 \cdot d_1^{r_1}} \right) \left(\frac{x+1}{e_2 \cdot d_2^{r_2}} \right) \left(\frac{a_mx-b_m}{D} \right) = \left(\frac{y}{d_1d_2}\right)^p, \] where $r_1, r_2 \in \{1,p-1\}$, $e_1, e_2 \mid c_m$, $D$ is some integer, $Y \mid y$, and the three factors on the left-hand side are integral and pairwise coprime. Now, if $r_i = p-1$ then we rewrite $1/d_i^{r_i}$ as $d_i / d_i^p$. It follows that \[ x = \frac{u_1}{v_1} z_1^p \; \text{ and } \; x+1 = \frac{u_2}{v_2} z_2^p, \] for some positive integers $u_i, v_i,$ and $z_i$. Moreover, we observe that \[ u_1v_1 \mid c_m b_m \; \text{ and } \; u_2v_2 \mid c_m (a_m+b_m). \] 
We then obtain the binomial Thue equation \begin{equation} \label{logthue} u_2v_1z_2^p - u_1v_2 z_1^p = v_1v_2, \end{equation} with $\max \{u_2v_1, u_1v_2, v_1v_2 \} \leq c_m^2 \cdot b_m \cdot (a_m+b_m)$. The proposition now follows by applying \citep[Theorem~2]{mignotte}, a result due to Mignotte obtained using linear forms in logarithms, where we have used $\lambda \geq \log(2)$ in this result, so that $7400 / \lambda < 10676$. 
\end{proof}

To give some indication of the magnitude of this bound, when $m = 6$ the bound we obtain is $55440$, and  when $m = 50$ the bound is $80372$. We note that for each triple $(A,B,C)$ we could provide a distinct bound on $p$ using the system (\ref{system}), and this bound will usually be smaller than the one obtained in Proposition \ref{LogBound}, but for simplicity we use a single bound for each value of $m$.

We will now list the possibilities for the triple $(A,B,C)$ for each value of $m$. In general, we write \[b_m = 2^t \cdot p_1^{r_1} \cdot q_1^{s_1} \quad a_m+b_m = p_2^{r_2} \cdot q_2^{s_2}, \] where $p_i$ and $q_i$ are odd primes that do not divide $c_m$, and $t, r_i, s_i$ are non-negative integers. We split into six cases dependent on $t = \ord_2(b_m)$.

\subsection*{Case 1: \texorpdfstring{$\ord_2(b_m) = 0$}{}}

Here,
\begin{align*} m \in \{ & 6, 8, 10, 12, 14, 16, 18, 20, 22, 24, 26, 28, 30, 32, 34, 36, 38, 40, 42, 44, 46,  \\ & 48, 50 \}.\end{align*} We have
\begin{itemize}
\item $b_m = p_1^{r_1} \cdot q_1^{s_1}$ for some $r_1 \in \{ 0,1,2 \}$, $s_1 \in \{ 0,1 \}$, and $p_1, q_1 \nmid c_m$ are prime;
\item $a_m+b_m = p_2^{r_2} \cdot q_2^{s_2}$ for some $r_2 \in \{ 1,2,3 \}$, $s_2 \in \{0,1\}$, and $p_2, q_2 \nmid c_m$ are prime.
\end{itemize}
Then \begin{align*} A & = 2^{\alpha_1} \cdot 3^{\beta_1} \cdot p_1^{\gamma_1} \cdot q_1^{\delta_1}, \\
B & = 2^{\alpha_2} \cdot 3^{\beta_2} \cdot p_2^{\gamma_2} \cdot q_2^{\delta_2},\\ 
C & = 3^{\beta_3} \cdot p_1^{p-\gamma_1} \cdot q_1^{p-\delta_1} \cdot p_2^{p-\gamma_2} \cdot q_2^{p-\delta_2}, \end{align*}
where \begin{align*} (\alpha_1,\alpha_2) & \in \{(0,1),(1,0) \}, & \\
(\beta_1, \beta_2, \beta_3) & \in \begin{cases} \{ (1,0,0),(0,1,0),(0,0,1) \} & \text{ if } c_m = 6, \\
 \{ (0,0,0) \} & \text{ if } c_m = 2, \end{cases} \\
\gamma_i & \in \{0,r_i,p-r_i \}, & \\
\delta_i & \in \{0,s_i,p-s_i\}.  & \\
\end{align*}
We note that if, say, $\gamma_1 = 0$, then we use the convention of removing the perfect $p$th power $p_1^p$ from $C$. We do this (in each case) to avoid introducing too many variables.

\begin{proof}[Proof of Case 1] As discussed earlier in this section, the basic idea is to divide both sides of (\ref{eq2}) by $p$th powers of $d_1$ and $d_2$ in order to obtain three pairwise coprime factors. We will work one prime at a time.

We start by dividing both sides of equation (\ref{eq2}) by $2$. If $x$ is even then $A$ will be exactly divisible by $2$ and both $B$ and $C$ will be odd. Otherwise, both $x+1$ and $a_mx-b_m$ are even. In this case we choose to divide $x+1$ by $2$ so that $B$ is even and $A$ and $C$ are odd.

Next, if $c_m = 6$, we divide both sides by $3$, and precisely one of $x$, $x+1$, and $a_mx-b_m$ will be divisible by $3$, so precisely one of $A$, $B$, and $C$ will be exactly divisible by $3$.

We now consider the prime $p_2$ and split into three cases depending on the value of $r_2$. The other primes ($p_1$, $q_1$, and $q_2$) are dealt with in the same manner.

\begin{enumerate}[(i)] 
\item Case $r_2 = 1$. If $p_2 \nmid d_2$, then our three factors will be pairwise coprime at $p_2$ (i.e. $p_2$ does not divide more than one of the three factors) and there is nothing more to do, so we assume that $p_2 \mid d_2$. Since $p_2 \parallel a_m+b_m$, we have that  $p_2 \parallel d_2$.

We then divide both sides of equation (\ref{eq2}) by $p_2^p$. If $p_2 \parallel x$, then $p_2^{p-1} \mid a_mx-b_m$, so $B$ will have a factor of $p_2$ and $C$ will have a factor of $p_2^{p-1}$. Otherwise, $p_2 \parallel a_mx-b_m$, so $C$ will have a factor of $p_2$ and $B$ will have a factor of $p_2^{p-1}$.

\item Case $r_2 = 2$. If $p_2 \nmid d_2$ then the three factors are pairwise coprime at $p_2$. If $p_2^2 \parallel d_2$, then after dividing by $p_2^p$, one of $B$ and $C$ will have a factor of $p_2^2$, and the other will have a factor of $p_2^{p-2}$. 

Next, if $p_2 \parallel d_2$, then one of $x+1$ and $a_mx-b_m$ will be divisible by $p_2^{p-1}$, and in particular by $p_2^2$. If $p_2^2 \mid x+1$, then since $p_2^2 \mid a_m+b_m$, we obtain $p_2^2 \mid a_mx-b_m$, contradicting $p_2 \parallel d_2$. We obtain a similar contradiction in the case $p_2^2 \mid a_mx-b_m$ and therefore conclude that $\ord_{p_2}(d_2) \ne 1$.

\item Case $r_2 = 3$. If $p_2 \nmid d_2$ then we argue as in the first two cases. Arguing as in Case (ii), we see that $\ord_{p_2}(d_2) \ne 1$ or $2$. Suppose $p_2^3 \parallel d_2$. We then divide both sides of equation (\ref{eq2}) by $p_2^p$. One of $B$ and $C$ will have a factor of $p_2^3$ and the other a factor of $p_2^{p-3}$.
\end{enumerate}

We repeat this process with the primes $p_1$, $q_1$, and $q_2$, until the factors are pairwise coprime. This gives the possibilities listed for the triple $(A,B,C)$. \end{proof}

\subsection*{Case 2: \texorpdfstring{$\ord_2(b_m) = 1$}{}}

Here, \[ m \in \{ 7, 11, 15, 19, 23, 27, 31, 35, 39, 47 \}. \] We have
\begin{itemize}
\item $b_m = 2 \cdot p_1^{r_1}$ for some $r_1 \in \{ 0,1 \}$, and $p_1 \nmid c_m$ is prime;
\item $a_m+b_m = p_2 \cdot q_2^{s_2}$, for some $s_2 \in \{ 0,1 \}$, and $p_2, q_2 \nmid c_m$ are prime.
\end{itemize}
Then \begin{align*} A & = 2^{\alpha_1} \cdot 3^{\beta_1} \cdot p_1^{\gamma_1}, \\
B & = 2^{\alpha_2} \cdot 3^{\beta_2} \cdot p_2^{\gamma_2} \cdot q_2^{\delta_2},\\ 
C & = 2^{\alpha_3} \cdot 3^{\beta_3} \cdot p_1^{p-\gamma_1} \cdot p_2^{p-\gamma_2} \cdot q_2^{p-\delta_2},\end{align*}
where \begin{align*} (\alpha_1,\alpha_2, \alpha_3) & \in \{(1,0,0),(0,1,0),(0,0,1) \}, & \\
(\beta_1, \beta_2, \beta_3) & \in \begin{cases} \{ (1,0,0),(0,1,0),(0,0,1) \} & \text{ if } c_m = 6, \\
 \{ (0,0,0) \} & \text{ if } c_m = 2, \end{cases} \\
\gamma_1 & \in \{0,r_1,p-r_1\}, &  \\
\gamma_2 & \in \{0,1,p-1\}, & \\
\delta_2 & \in \{0,s_2,p-s_2 \}. &
\end{align*}

\begin{proof}[Proof of Case 2] For primes away from $2$, we argue in exactly the same way as in Case 1. We only need to consider what happens at the prime $2$. If $x$ is odd, then $a_mx-b_m$ is also odd, and we simply divide equation (\ref{eq2}) by $2$ so that $B$ is exactly divisible by $2$. Since $a_m+b_m$ is odd, there is nothing more to do.

Now we suppose $x$ is even. Since $\ord_2{b_m} = 1$, we must have $2 \parallel  d_1$. We divide both sides of equation (\ref{eq2}) by $2^{p+1}$. Either $2 \parallel x$ and $2^p \mid a_mx-b_m$, or $2^p \mid x$ and $ 2 \parallel a_m x- b_m$. Since we may absorb any $p$th powers, we have  $(\alpha_1,\alpha_3) = (0,1)$ or $(1,0)$.
\end{proof}

\subsection*{Case 3: \texorpdfstring{$\ord_2(b_m) = 2$}{}}

Here, \[m \in \{ 9,17,25,33,41,49  \}. \] We have
\begin{itemize}
\item $b_m = 4 \cdot p_1^{r_1}$ for some $r_1 \in \{ 0,1 \}$, and $p_1 \nmid c_m$ is prime;
\item $a_m+b_m = p_2^{r_2} \cdot q_2^{s_2}$, for some $r_2 \in \{1,2\}$, $s_2 \in \{ 0,1 \}$, and $p_2, q_2 \nmid c_m$ are prime.
\end{itemize}
Then \begin{align*} A & = 2^{\alpha_1} \cdot 3^{\beta_1} \cdot p_1^{\gamma_1}, \\
B & = 2^{\alpha_2} \cdot 3^{\beta_2} \cdot p_2^{\gamma_2} \cdot q_2^{\delta_2},\\ 
C & = 2^{\alpha_3} \cdot 3^{\beta_3} \cdot p_1^{p-\gamma_1} \cdot p_2^{p-\delta_2} \cdot q_2^{p-\delta_2},\end{align*}
where \begin{align*} (\alpha_1,\alpha_2, \alpha_3) & \in \{(0,1,0),(2,0,p-1),(p-1,0,2) \}, & \\
(\beta_1, \beta_2, \beta_3) & \in \begin{cases} \{ (1,0,0),(0,1,0),(0,0,1) \} & \text{ if } c_m = 6, \\
 \{ (0,0,0) \} & \text{ if } c_m = 2, \end{cases} \\
\gamma_1 & \in \{0,r_1,p-r_1\}, &\\
\gamma_2 & \in \{0,r_2,p-r_2\}, &\\
\delta_2 & \in \{0,s_2,p-s_2 \}. &
\end{align*}

\begin{proof}[Proof of Case 3]
We will only consider what happens at the prime $2$ in the case $x$ even. The other primes and the case when $x$ is odd can be dealt with as in the proofs of cases 1 and 2.

We first claim that $2^2 \parallel d_1$. If not, then we must have $2 \parallel d_1$, and then either $2^2 \mid x$ or $2^2 \mid (a_mx-b_m)$. Since $2^2 \mid b_m$, we will have that $2^2 \mid x$ and $2^2 \mid a_mx-b_m$, a contradiction, proving the claim.

We now divide both sides of equation (\ref{eq2}) by $2^{p+1}$. One of $x$ and $a_mx-b_m$ will be exactly divisible by $2^2$ and the other will be divisible by $2^{p-1}$, and so $(\alpha_1, \alpha_3) = (2,p-1)$ or $(p-1,2)$. 
\end{proof}

\subsection*{Case 4: \texorpdfstring{$\ord_2(b_m) = 3$}{}}

Here, \[m \in \{ 13,29,45 \}. \] We have
\begin{itemize}
\item $b_m = 8 \cdot p_1^{r_1}$ for some $r_1 \in \{ 0,1 \}$, and $p_1 \nmid c_m$ is prime;
\item $a_m+b_m = p_2$, and $p_2 \nmid c_m$ is prime.
\end{itemize}
Then \begin{align*} A & = 2^{\alpha_1} \cdot 3^{\beta_1} \cdot p_1^{\gamma_1}, \\
B & = 2^{\alpha_2} \cdot 3^{\beta_2} \cdot p_2^{\gamma_2},\\ 
C & = 2^{\alpha_3} \cdot 3^{\beta_3} \cdot p_1^{p-\gamma_1} \cdot p_2^{p-\gamma_2}, \end{align*}
where \begin{align*} (\alpha_1,\alpha_2, \alpha_3) & \in \{(0,1,0),(3,0,p-2),(p-2,0,3) \}, & \\
(\beta_1, \beta_2, \beta_3) & \in \begin{cases} \{ (1,0,0),(0,1,0),(0,0,1) \} & \text{ if } c_m = 6, \\
 \{ (0,0,0) \} & \text{ if } c_m = 2, \end{cases} \\
\gamma_1 & \in \{0,r_1,p-r_1\}, & \\
\gamma_2 & \in \{0,1,p-1\}. &
\end{align*}
When $p = 3$, we must also consider the cases $(\alpha_1, \alpha_2, \alpha_3) = (2,0,2)$, with $\beta_i$ and $\gamma_i$ varying as above.

\begin{proof}[Proof of Case 4] As in Case 3, we will only consider $x$ even and the prime $2$. If $2 \parallel d_1$ then we obtain a contradiction as in Case 3. Suppose $2^2 \parallel d_1$. Then if $2^3 \mid x$ or $a_mx-b_m$ we obtain a contradiction as before, so we must have $2^2 \parallel x, a_mx-b_m$. The valuation at $2$ of the left-hand side of equation (\ref{eq2}) is thus $4$, and it is $p\ord_2(y)+1$ for the right-hand side of equation (\ref{eq2}). This forces $p = 3$ and $(\alpha_1,\alpha_3)  =(2,2)$.

Finally, if $2^3 \parallel \gcd(x,a_mx-b_m)$ then we divide by $2^{p+1}$ and one of $A$ and $C$ will have a factor of $2^3$, and the other a factor of $2^{p-2}$.
\end{proof}

\subsection*{Case 5: \texorpdfstring{$\ord_2(b_m) = 4$}{}}

Here, \[m = 21.\] We have
\begin{itemize}
\item $b_m = 16$ and $c_m = 6$;
\item $a_m+b_m = p_2 \cdot q_2$, and $p_2, q_2 \nmid c_m$ are prime.
\end{itemize}
Then \begin{align*} A & = 2^{\alpha_1} \cdot 3^{\beta_1}, \\
B & = 2^{\alpha_2} \cdot 3^{\beta_2} \cdot p_2^{\gamma_2} \cdot q_2^{\delta_2},\\ 
C & = 2^{\alpha_3} \cdot 3^{\beta_3} \cdot p_2^{p-\gamma_2} \cdot q_2^{p-\delta_2}, \end{align*}
where \begin{align*} (\alpha_1,\alpha_2, \alpha_3) & \in \{(0,1,0),(4,0,p-3),(p-3,0,4) \}, \\
(\beta_1, \beta_2, \beta_3) & \in \{ (1,0,0),(0,1,0),(0,0,1) \}, \\
\gamma_2, \delta_2 & \in \{0,1,p-1\}.
\end{align*}
When $p = 3$ or $p = 5$ we must also consider the cases \[(\alpha_1, \alpha_2, \alpha_3) = \left((p+1)/2, 0, (p+1)/2 \right), \] with $\beta_i, \gamma_2,$ and $\delta_2$ varying as above.

\begin{proof}[Proof of Case 5]
We are in a very similar set-up to Case 4. If $2^2 \parallel d_1$ then we must have $p = 3$ and $(\alpha_1,\alpha_3) = (2,2)$. If $2^3 \parallel d_1$, then by comparing valuations at $2$ on each side of equation (\ref{eq2}), we have $6 = p \ord_2(y) + 1$, which forces $p = 5$ and $(\alpha_1,\alpha_3) = (3,3)$.
Next, if $2^4 \parallel d_1$, then we divide through by $2^{p+1}$ and argue as in previous cases.
\end{proof}

\subsection*{Case 6: \texorpdfstring{$\ord_2(b_m) = 5$}{}}

Here, \[ m = 37. \] We have
\begin{itemize}
\item $b_m = 32$ and $c_m = 6$;
\item $a_m+b_m = p_2$, and $p_2 \nmid c_m$ is prime.
\end{itemize}
Then \begin{align*} A & = 2^{\alpha_1} \cdot 3{^\beta_1}, \\
B & = 2^{\alpha_2} \cdot 3^{\beta_2} \cdot p_2^{\gamma_2},\\ 
C & = 2^{\alpha_3} \cdot 3^{\beta_3} \cdot p_2^{p-\gamma_2},
\end{align*}
where \begin{align*} (\alpha_1,\alpha_2, \alpha_3) & \in \begin{cases} \{(0,1,0),(5,0,p-4),(p-4,0,5) \} & \text{ if } p > 3, \\ \{(0,1,0),(2,0,2) \} & \text{ if } p = 3,
\end{cases}
 \\
(\beta_1, \beta_2, \beta_3) & \in \{ (1,0,0),(0,1,0),(0,0,1) \}, & \\
\gamma_2 & \in \{0,1,p-1\}. &
\end{align*}
When $p = 3$, $p = 5$, or $p = 7$ we must also consider the cases \[(\alpha_1, \alpha_2, \alpha_3) = \left((p+1)/2, 0, (p+1)/2 \right), \] with $\beta_i$ and  $\gamma_2$ varying as above.

\begin{proof}[Proof of Case 6] This is almost identical to the argument given in Case 5, apart when $2^5 \parallel d_1$ and $p = 3$. In this case, dividing by $2^{p+1} = 2^4$ is not enough to make the factors coprime. Instead, we divide through by $2^{2p+1} = 2^7$. One of $A$ and $C$ will have a factor of $2^5$, and the other a factor of $2^2$. After removing the perfect cube $2^3$ from $2^5$, we have $(\alpha_1, \alpha_3) = (2,2)$.
\end{proof}

\section{The Local Method}

We consider the system (\ref{system}) of Thue equations for a fixed value of $m$ and triple $(A,B,C)$. We will start by considering this system mod $\ell$, for many auxilliary primes $\ell$ to try and obtain a contradiction; since if the system of equations has no local solution then it will certainly not have a global solution. When the system of equations does not have a (global) solution, we found this method to be extremely effective (as we see below). The strategy we present here is used with a single binomial Thue equation in \citep[p.~492]{local}. 

Fix a prime $p > 2$. We search for a prime $\ell$ such that $\ell = 2kp+1$ for some $k \geq 1$ (i.e. $\ell \equiv 1 \pmod{p}$), such that $\ell \nmid ABC$, and for which the system of equations has no solution mod $\ell$. If we can find such an $\ell$, then we have obtained a contradiction. The reason for choosing $\ell$ of this form is that we have, for each $i \in \{1,2,3\}$, either $\ell \mid y_i$, or \[ (y_i^p)^{2k} = y_i^{\ell-1} \equiv 1 \pmod{\ell}. \] In particular, $y_i^p \in \mu_{2k}(\mathbb{F}_\ell) \cup \{0\},$ where $ \mu_{2k}(\mathbb{F}_\ell)  = \{ \alpha \in \mathbb{F}_\ell : \alpha^{2k}  = 1 \}$. We therefore only have $2k+1$ possibilities for $y_i^p \pmod{\ell}$, and moreover the set $\mu_{2k}(\mathbb{F}_\ell)$ can be computed extremely quickly using a primitive root modulo $\ell$. Indeed, if $g$ is a primitive root modulo $\ell$, then  \[ \mu_{2k}(\mathbb{F}_\ell) = \{({g^p})^r : 0 \leq r \leq 2k-1 \}.\]

For each triple $(A,B,C)$, we searched for a prime $\ell$ by testing with $1 \leq k \leq 150$. For $p>5$, with $p$ less than the prime bound for $m$ obtained in Proposition \ref{LogBound}, apart from the cases where we have a global solution, and a single case when $p = 7$, we succeeded in obtaining a contradiction.

When $p=3$ or $p = 5$, the method sometimes fails even when there is no global solution. In these cases, as $p$ is small we can simply solve the two Thue equations using \texttt{Magma} and verify whether we have a solution $(y_1,y_2,y_3)$ with $y_1,y_2 > 0$ (since $x > 0$). As mentioned above, the local method also fails for $p = 7$ in a single case. This is for the case $m = 21$ and $(A,B) = (2^4 \cdot 3, 1)$. Here we also simply solve the corresponding Thue equations directly to conclude there are no non-zero solutions.

For certain triples, the local method will fail \emph{for all values of $p$} as we have a global solution for all $p$. There are three cases when this happens.

\begin{enumerate}[(I)]
\item $A = 1, B = 2$, and $a_m-C = b_m$ . Here we have a global solution $(y_1,y_2,y_3) = (1,1,1)$ for all $p$, which comes from the solution $x= y = 1$ to our original equation. However, in this case, our first Thue equation is \[ 2y_2^p - y_1^p = 1.\] Applying a well-known result of Darmon and Merel \citep[Main~Theorem]{ll2}, we see that $y_1 = y_2 = 1$ for all $p$, so $x = 1$.

\item $A = 1$ and $C = a_m + b_m$. This admits the solution $(y_1,y_2,y_3) = (-1,0,-1)$.

\item $B = 1$ and $C = b_m$. This admits the solution $(y_1,y_2,y_3) = (0,1,-1)$. 
\end{enumerate}

In cases (II) and (III) we must use a different strategy. We use Theorem \ref{ThmD} (stated below) together with the modular method.

\section{The Modular Method}

It remains to deal with cases (II) and (III), outlined in Section 4, for each $6 \leq m \leq 50$. In each case, we have $A = 1$ or $B = 1$, and this leads to an equation of the form \begin{equation}\label{eqD} z_1^p - Dz_2^p = 1 \end{equation} for integers $z_1$ and  $z_2$, where $D = A$ if $B = 1$, and $D = B$ if $A = 1$. The following result of Bartolom\'e and Mih{\u{a}}ilescu will be very helpful. 

\begin{theorem}[{\citep[Theorem~1.3]{bartmihai}}]\label{ThmD} Let $D>1$ and and let $p$ be an odd prime satisfying \[ \gcd(\mathrm{Rad}(\varphi(D),p) = 1. \] Suppose $z_1$ and $z_2$ are integers satisfying equation (\ref{eqD}) with $ \abs{z_2} > 1$. Then either $(z_1,z_2,D,p) = (18,7,17,3)$ or $p > 163 \cdot 10^{12}$.
\end{theorem}

Here, $\varphi$ denotes Euler's totient function, and $\mathrm{Rad}(\varphi(D))$ denotes the product of all primes dividing $\varphi(D)$. We note that the theorem stated in \citep{bartmihai} goes on to give further constraints in the case $p > 163 \cdot 10^{12}$.

Since the prime bound obtained in Proposition \ref{LogBound} is (much) smaller than $163 \cdot 10^{12}$ in each case, Theorem \ref{ThmD} reduces our problem to only needing to consider finitely many small primes in each case; namely the odd prime factors of $\varphi(D)$.

When $ p = 3$, $5$, or $7$ occurs as a factor of $\varphi(D)$, we found it simplest to directly solve the relevant Thue equations with \texttt{Magma}. For $p \geq 11$, in the cases when $D$ has at most two prime factors and its largest prime factor is $<30$, we may apply \citep[Theorem~1]{GP} to immediately rule out the existence of solutions. In the case $(m, D, p) = (49, 2^{10} \cdot 3 \cdot 11^{10}, 11)$ we directly solve the corresponding Thue equation with \texttt{Magma} (although this case could also be dealt with using similar techniques to those we present below). 

For each remaining case, $\ord_2(D) = 1$ and we will now employ some of the local arguments of Section 4 in combination with the modular method. We start by seeing how one may associate, following standard recipes (see \citep{ppp} for example), a Frey curve to equation (\ref{eqD}). We rewrite equation (\ref{eqD}) as \begin{equation}\label{ppp} -1 - Dz_2^p +z_1^p = 0,  \end{equation} and we assume that $p \geq 11$ and $\ord_2(D) = 1$, since this will be our set-up. The Frey curve we associate to this equation is \[E: \; Y^2 = X(X+1)(X-Dz_2^p). \] The conductor, $N$, of $E$ is then given by \[ N = \begin{cases} 2 \cdot  \mathrm{Rad}_2(Dz_1z_2) & \text{ if } \; \; 2 \mid z_2, \\ 2^5 \cdot  \mathrm{Rad}_2(Dz_1z_2) & \text{ if }  \; \;2 \nmid z_2.  \end{cases} \] Here, $\mathrm{Rad}_2(Dz_1z_2)$ denotes the product of all \emph{odd} primes dividing $Dz_1z_2$. We write $\overline{\rho}_{E,p}$ for the mod $p$ Galois representation of $E$. Applying standard level-lowering results, we obtain that \[ \overline{\rho}_{E,p} \sim \overline{\rho}_{f,\mathfrak{p}}, \] for $f$ a newform at level $N_p$, where \[N_p = \begin{cases} 2 \cdot \mathrm{Rad}_2(D) & \text{ if } \; \; 2 \mid z_2, \\ 2^5 \cdot  \mathrm{Rad}_2(D) & \text{ if }  \; \;2 \nmid z_2,  \end{cases} \] and $\mathfrak{p}$ a prime above $p$ in the coefficient field of $f$.

We are now in a position to complete the proof of Theorem \ref{Mainthm}.

\begin{proof}[Proof of Theorem \ref{Mainthm}] It remains to deal with the odd primes dividing $\varphi(D)$ in the remaining cases. Suppose we are in one of these cases, and let $(y_1,y_2,y_3)$ be a non-zero solution to the system (\ref{system}) of Thue equations. By rewriting $y_i^p$ as $-(-y_i)^p$ if necessary, we obtain an equation of the same form as (\ref{ppp}). As described above, we attach a Frey curve $E$ to this equation, and level lower so that $\overline{\rho}_{E,p} \sim \overline{\rho}_{f,\mathfrak{p}}$, for $f$ a newform at level $2 \cdot \mathrm{Rad}_2(D)$ or $2^5 \cdot \mathrm{Rad}_2(D)$.

 Now, if $\ell \mid y_1y_2$ is a prime, then it must be a prime of multiplictive reduction for $E$, and by comparing traces of Frobenius, we have \[\ell+1 \equiv \pm c_\ell(f) \pmod{\mathfrak{p}},\] where $c_\ell(f)$ denotes the $\ell$th Fourier coefficient of the newform $f$. It follows that\begin{equation}\label{traces} p \mid \mathrm{Norm}((\ell+1)^2 - c_\ell(f)^2) \end{equation}

We now search for a prime $\ell \nmid D$ with $\ell \equiv 1 \pmod{p}$, for which the system of Thue equations (\ref{system}) has a unique solution  mod ${\ell}$, \emph{and} for which (\ref{traces}) does not hold. If the system has a unique solution mod $\ell$, then this solution must be the reduction mod $\ell$ of the known global solution, for which $y_1y_2 = 0$, so either $y_1 \equiv 0 \pmod{\ell}$ or   $y_2 \equiv 0 \pmod{\ell}$. So $\ell \mid y_1y_2$, and we have therefore obtained a contradiction if (\ref{traces}) does not hold. For each newform $f$ in each case we were able to find such a prime $\ell$, apart from the cases listed in Table \ref{TabSturm}.

For the remaining newforms in Table \ref{TabSturm}, we find that for any prime $q \nmid 2D$ that we test, \begin{equation}\label{strm} p \mid \mathrm{Norm}(q+1 - c_q(f)). \end{equation} This suggests that the representation $\overline{\rho}_{f,\mathfrak{p}}$ is reducible, which would be a contradiction. We proceed by applying \citep[Proposition~2.2]{tripexp} to the newform $f$. We obtain that $p \mid \# E(\mathbb{F}_q)$ for any prime $q \nmid D$, and so $E$ must have a rational subgroup of order $p$, a contradiction since $p \geq 11$.
\begingroup 
\renewcommand*{\arraystretch}{1.5}
\begin{table}[ht!]
\begin{center} \small
\begin{tabular}{ |c|c|c| } 
 \hline
 $m$ & $p$ & $f$  \\
 \hline 
 $15$ & $11$ & 138.2.a.d \\ \hline 
 $27$ & $23$ & 282.2.a.e \\ \hline
 $28$ & $11$ & 138.2.a.d \\ \hline
 $30$ & $13$ & 318.2.a.g \\ \hline
 $33$ & $29$ & 354.2.a.h \\ \hline
 $37$ & $11$ & 402.2.a.g \\ \hline
 $43$ & $13$ & 474.2.a.e \\ \hline
 $45$ & $41$ & 498.2.a.g \\ \hline
 $48$ & $11$ & 534.2.a.f \\ \hline 
\end{tabular}
\caption{\label{TabSturm}\normalsize Remaining newforms. We use the notation of the LMFDB \citep{lmfdb}.}
\end{center}
\end{table} 
\endgroup
\end{proof}

\bibliographystyle{plainnat}

\Addresses

\end{document}